\documentclass[pdflatex,sn-mathphys-num]{sn-jnl}

\usepackage{graphicx}%
\usepackage{multirow}%
\usepackage{amsmath,amssymb,amsfonts}%
\usepackage{amsthm}%
\usepackage{mathrsfs}%
\usepackage[title]{appendix}%
\usepackage{xcolor}%
\usepackage{textcomp}%
\usepackage{manyfoot}%
\usepackage{booktabs}%
\usepackage{algorithm}%
\usepackage{algorithmicx}%
\usepackage{algpseudocode}%
\usepackage{listings}%
\usepackage{hyperref}

\theoremstyle{thmstyleone}%
\newtheorem{theorem}{Theorem}
\newtheorem{lemma}{Lemma}
\theoremstyle{thmstyletwo}%
\newtheorem{example}{Example}%

\theoremstyle{thmstylethree}%
\newtheorem{definition}{Definition}%

\begin{document}

\title[Borsuk-Ulam Type Theorems and Mountain Climbing Problem]{Borsuk-Ulam Type Theorems and Mountain Climbing Problem}


\author*[1]{\fnm{Ilya M.} \sur{Shirokov}}\email{shirokov.im@phystech.edu}

\author[1,2]{\fnm{Andrey V.} \sur{Malyutin}}

\author[2]{\fnm{Alisa} \sur{Volkova}}

\affil*[1]{\orgname{St.\ Petersburg Department of Steklov Mathematical Institute},
\orgaddress{\street{Fontanka, 27}, \city{St.\ Petersburg}, \postcode{191023}, \country{Russia}}}

\affil[2]{\orgname{St.\ Petersburg State University},
\orgaddress{\street{Universitetskaya Emb., 13B}, \city{St.\ Petersburg}, \postcode{199034}, \country{Russia}}}


\abstract{In this paper, we present a new qualitative extension of the Hopf theorem (and a generalization of Borsuk-Ulam theorem), concerning continuous maps $f$ from a compact Riemannian manifold $M$ of dimension $n$ to $\mathbb{R}^n$. We remove the assumption of a Riemannian structure and instead consider closed triangulable manifolds $M$ equipped with a topological notion of 'distant' points. We show that for any continuous map $f \colon M \to \mathbb{R}^n$, there exists a connected component in the space of $f$-neighbors (where a pair of points $a, b$ are $f$-neighbors if $f(a) = f(b)$) that contains both a pair of 'distant' points and a pair of identical points. This result yields further consequences for Lusternik-Schnirelmann and Tucker-type theorems, as well as a multidimensional extension of the mountain-climbing lemma, which in the special case of the standard Euclidean $2$-sphere, may be stated informally as follows. For any continuous distribution of temperature and pressure on Earth (assumed time-independent), there exists a pair of antipodal points with identical values such that travelers starting from these points can move and meet while, at each moment of their journey, experiencing matching 'climatic conditions' up to an arbitrarily small constant.}




\maketitle

\section{Introduction}\label{sec1}


This work was initiated as a continuation of our previous results on generalizations of the Hopf theorem, concerning continuous maps $f$ of a compact Riemannian manifold $M$ of dimension $n$ to Euclidean space $\mathbb{R}^n$  \cite{Previous, Previous1}. However the method and ideas, developed in the course of this investigation turned out to be applicable to a much wider class of the so-called `point-pair type theorems', where a pair of points plays a central role: Borsuk-Ulam, Tucker and Lusternik–Schnirelmann theorems, and mountain-climbing lemma. The common thread of all our results follows the principle that a set that is `large' in some sense is connected to a set that is `small' in some sense through sets satisfying a given relation. 

To describe our results in more detail, we start with the Hopf theorem:

\begin{theorem}[H. Hopf \cite{Hopf}]
    Let $n$ be a positive integer, let $M$ be a compact Riemannian manifold of dimension $n$, and let $f\colon M \to \mathbb{R}^n$ be a continuous map. Then for any prescribed $\delta > 0$, there exists a pair $\{x,y\} \in M \times M$ such that $f(x) = f(y)$ and $x$ and $y$ are joined by a geodesic of length~$\delta$.
\end{theorem}

In \cite{Previous}, we proved that the set of pairs of points satisfying the Hopf theorem for a fixed value of~$\delta$ (under the assumption that there are no $\delta$-conjugate points, that is, there is no pair of points joined by an infinite number of
geodesics of length~$\delta$) is uncountable. Understanding the topological structure of this set was the main driving force for~\cite{Previous1}, where we proved that for a piecewise linear map~$f$ in general position of a triangulation of a $2$-sphere into the Euclidean plane~$\mathbb{R}^2$, there exists a path in the space of $f$-neighbors\footnote{A pair of points $(a,b) \in M \times M$ is called $f$-neighbors if $f(a)=f(b)$.} connecting a pair of `antipodal' points with a pair of identical points. In this setting, `antipodal' and identical points play the roles of `large' and `small' sets, respectively, and the connecting relation is that pairs are $f$-neighbors.

In this work, we first strengthen this result for the case of $n$-dimensional closed triangulable manifolds. This gives a topological version of the Hopf theorem: instead of relying on the Riemannian structure or a metric, we introduce a topological notion of `distant' pairs of points on~$M$. For the case of $n$-spheres, this yields an $n$-dimensional extension of the mountain-climbing lemma. It is worth noting that several authors have dedicated entire papers to different proofs of the original one-dimensional mountain-climbing lemma (see, for example, \cite{Whittaker}, \cite{Huneke}, \cite{Keleti}, \cite{Lopez}), sometimes even without citing each other. Our approach provides a more general perspective on this lemma from the point of view of Borsuk--Ulam type theorems.

The topological version of the Hopf theorem can also be directly applied to extend the Tucker theorem (stated for antipodal triangulations of spheres) and the Lusternik--Schnirelmann theorem. Within our extension framework, for the Lusternik--Schnirelmann theorem, the relevant relation on a pair of points is that they lie almost in the same element of the covering. For the Tucker theorem, a pair of points either has opposite labels or consists of midpoints of complementary edges (see Sections~3 and~4 for a detailed exposition). An extension of Sperner's theorem will be considered in a subsequent paper.

An interesting future direction is to extend the notion of Borsuk–Ulam type spaces (spaces with a free involution that satisfy the Borsuk–Ulam theorem), introduced in~\cite{Musin}, and investigate the class of Hopf-type spaces: spaces with a free involution for which the topological version of the Hopf theorem holds. In subsequent work, we plan to investigate how these two classes of spaces differ.

The paper is organized as follows. First, in Section~2 we provide a topological extension of the Hopf theorem for closed triangulable manifolds, including a corollary for the mountain-climbing lemma. Next, in Sections~3 we present generalizations of the Lusternik--Schnirelmann and Tucker theorems.

\section{Topological extension of the Hopf theorem.}

First, we define a notion of 'distant' points in a topological sense.

\begin{definition}[Distant points]
Let $M$ be a closed topological manifold of dimension $n \ge 1$.
Let $N \subset M \times M$ be an open neighborhood of the diagonal $\Delta_M = \mathrm{diag}(M\times M)$, and set $K = (M \times M)\setminus N$. Suppose there exists a $\mathbb{Z}_2$-equivariant retraction of $N$ onto $\Delta_M$ \footnote{By $\mathbb{Z}_2$-action on $M \times M$ we mean the reflection about the diagonal of $M\times M$.}.  In this case we say that a \emph{distant relation} is given on $M$, and call any pair $(a,b)\in K$ a pair of \emph{distant} points.
\end{definition}

\begin{example}
    Let $M$ be a Riemannian manifold. We call two points $a$ and $b$ in $M$ distant, if they are connected by more than one minimizing geodesic. Retraction of $N$ onto $\Delta_M$ is defined by uniform contraction of corresponding minimizing geodesics to their midpoints. This defines a distant relation on $M$.
\end{example}
Let $M$ be a closed triangulable manifold of dimension $n \geq 1$. \footnote{Recall that a topological manifold is called triangulable if it admits a triangulation. A triangulation of a topological space is a pair $(\mathcal{K}, h)$ of an (abstract) simplicial complex $\mathcal{K}$ together with homeomorphism $h\colon |\mathcal{K}| \to M$ from its geometric realization $|\mathcal{K}|$ to $M$. By abuse of notation, we sometimes identify the triangulation with $\mathcal{K}$ and leave $h$ implicit.} Let $(\mathcal{K}, h)$ be a triangulation of $M$, and let $f\colon |\mathcal{K}| \to \mathbb{R}^n$ be a \emph{piecewise linear} map, that is, a continuous map which is affine on each simplex of a triangulation of $M$. Next, we refer to such maps simply as piecewise linear maps of $M$, and we say 'simplex of $M$', meaning a simplex of the corresponding triangulation, whenever it is clear which triangulation is used. 

By Theorem 2.14 of \cite{Peace-Wise}, there exist simplicial subdivisions $\mathcal{K}'$ of $\mathcal{K}$ and $\mathcal{M}'$ of $f(|\mathcal{K}|)$ such that $f\colon |\mathcal{K}'| \to |\mathcal{M}'|$ is simplicial. For a given piecewise linear map $f$ this defines an induced by $f$ triangulation of $M$, which we denote by $(\mathcal{K}_f, h)$. 

We recall that two points $a$ and $b$ of $M$ are called $f$-neighbors if $f(a) = f(b)$. \footnote{For convenience, we do not require that $a$ and $b$ are distinct.} 
\begin{definition}[Complex of $f$-neighbors]
    Let $M$ be a closed triangulable manifold of dimension $n \geq 1$, and let $f\colon M \to \mathbb{R}^n$ be a piecewise linear map in general position. For any two distinct simplices $\mathcal{A}$ and $\mathcal{B}$ of $\mathcal{K}_f$, whose images coincide, we define two simplices of corresponding $f$-neighbors in $M \times M$, which we denote by $[\mathcal{A}, \mathcal{B}]$ and $[\mathcal{B}, \mathcal{A}]$. We construct the complex of $f$-neighbors $\mathfrak{F}_f$ as a family of simplices, and their faces in a natural way. Notice that $\mathfrak{F}_f$ coincides with the closure (in $M \times M$) of the set of $f$-neighbors realizing nonzero distances (for any metric $d$ on $M$ compatible with the topology of $M$) \footnote{$f$-neighbors $a$ and $b$ realize distance $\delta \geq 0$ if $d(a,b) = \delta$.}. 
\end{definition}

\begin{lemma}
    Let $M$ be a closed triangulable manifold of dimension $n \geq 1$, and let $f\colon M \to \mathbb{R}^n$ be a piecewise linear map in general position. Then the associated complex of $f$-neighbors $\mathfrak{F}_f$ is a pseudomanifold, that is, any $(n-1)$-simplex of $\mathfrak{F}_f$ belongs to exactly two $n$-simplices of $\mathfrak{F}_f$.
\end{lemma}

\begin{proof}
    Let $[\mathcal{A}, \mathcal{B}]$ be an $n$-simplex in $\mathfrak{F}_f$ and let $\mathcal{F}$ be an $(n-1)$-face of $[\mathcal{A}, \mathcal{B}]$. There are the $(n-1)$-faces $\mathcal{F}_{\mathcal{A}}$ and $\mathcal{F}_{\mathcal{B}}$ in simplices $\mathcal{A}$ and $\mathcal{B}$ respectively such that $\mathcal{F} = [\mathcal{F}_{\mathcal{A}}, \mathcal{F}_{\mathcal{B}}]$. We have two cases:
    \begin{enumerate}
        \item If $\mathcal{F}_{\mathcal{A}}$ coincides with $\mathcal{F}_{\mathcal{B}}$, then simplices $[\mathcal{A}, \mathcal{B}]$ and $[\mathcal{B}, \mathcal{A}]$ share the face $\mathcal{F}$. 
        \item If $\mathcal{F}_{\mathcal{A}}$ and  $\mathcal{F}_{\mathcal{B}}$ are distinct, denote by $\mathcal{A}_1$ and $\mathcal{B}_1$ the $n$-simplices adjacent to faces $\mathcal{F}_{\mathcal{A}}$ and  $\mathcal{F}_{\mathcal{B}}$ respectively, besides $\mathcal{A}$ and $\mathcal{B}$. It follows that simplices $\mathcal{A}, \mathcal{A}_1, \mathcal{B}, \mathcal{B}_1$ are all distinct. 
        Since $f$ is in general position, we have two subcases:

        \begin{enumerate}[label=(\alph*)]
            \item `no foldings': $f(\mathcal{A}_1)$ coincides with $f(\mathcal{B}_1)$, but not with $f(\mathcal{A})$ and $f(\mathcal{B})$. In this subcase simplices $[\mathcal{A}, \mathcal{B}]$ and $[\mathcal{A}_1, \mathcal{B}_1]$ share the face $\mathcal{F}$. 
            \item `one folding': either $f(\mathcal{A}_1)$ or $f(\mathcal{B}_1)$ (but not both) coincides with $f(\mathcal{A})$ and $f(\mathcal{B})$. In this subcase either simplices $[\mathcal{A}, \mathcal{B}]$ and $[\mathcal{A}_1, \mathcal{B}]$, or simplices $[\mathcal{A}, \mathcal{B}]$ and $[\mathcal{A}, \mathcal{B}_1]$, share the face $\mathcal{F}$.
        \end{enumerate}
    \end{enumerate}

    We observe that in each case only the specified pair of simplices share the face $\mathcal{F}$. 
    
\end{proof}

\begin{theorem}
     Let $M$ be a closed triangulable manifold of dimension $n \geq 1$. And let $f\colon M \to \mathbb{R}^n$ be a piecewise linear map in general position. Then for any distant relation on $M$, there exists a path of $f$-neighbors in $\mathfrak{F}_f$ that connects a pair of distant points with a pair of identical points. 
\end{theorem}

\begin{proof}
    Let $pr \colon M \times M \to M$ be a projection. Observe that there exists a component $\mathcal{L}$ of $\mathfrak{F}_f$ such that $pr \colon \mathcal{L} \to M$ has degree $1$ and $\mathcal{L}$ intersects the diagonal of $M \times M$. Indeed, since $f$ is in general position there is a pair of adjacent $n$-simplices $\mathcal{A}$ and $\mathcal{B}$ in the induced by $f$ triangulation of $M$ such that $f(\mathcal{A}) = f(\mathcal{B})$ (`folding') and there are no other $n$-simplices in $M$ whose images coincide with $f(\mathcal{A})$ and $f(\mathcal{B})$. Thus we can take $\mathcal{L}$ as a component of $\mathcal{F}_f$ that contains $[\mathcal{A}, \mathcal{B}]$. 

    Suppose next that $\mathcal{L}$ doesn't contain a pair of distant points. Then there is a homotopy $r \colon \mathcal{L} \times [0,1] \to M \times M$, which takes $\mathcal{L}$ onto the diagonal $\mathrm{diag}(M\times M)$. This defines the homotopy $pr\circ r \colon \mathcal{L} \times [0,1] \to M$. Since $pr \circ r(1)$ has an even degree, this is a contradiction. 
    
\end{proof}

We now prove our main theorem by passing to the limit. 

\begin{theorem}
     Let $M$ be a closed triangulable manifold of dimension $n \geq 1$. And let $f\colon M \to \mathbb{R}^n$ be a continuous map. Then for any distant relation on $M$, there exists a connected component of $f$-neighbors that contains both a pair of distant points and a pair of identical points. 
\end{theorem}

\begin{proof}
    Let $(\mathcal{K}, h)$ be a triangulation of $M$. Construct a refining sequence of barycentric subdivisions $\{\mathcal{K}_n\}_{n=1}^{\infty}$ and piecewise linear maps in general position $\{f_{n} \colon \mathcal{K}_n \to \mathbb{R}^n\}_{n=1}^{\infty}$, which converges to $f$. \footnote{For instance, let $f_n$ coincide with $f$ on the vertices of $\mathcal{K}_n$, and extend it affinely to $\mathcal{K}_n$. To make $f_n$ a map in general position, a small perturbation may be needed.} By Theorem 1, we obtain a sequence of paths of $f_n$-neighbors in $M \times M$, which we denote by $\{\gamma_n \}_{n=1}^{\infty}$. Denote by $\Gamma$ the limit set, that is, the set of all accumulation points of the union of paths $\gamma_n$. Let $\Gamma_1$ be a connected component of $\Gamma$. It's straightforward to check that $\Gamma_1$ contains limit points of sequences $\{\gamma_n(t)\}_{n=1}^{\infty}$ for any $t \in [0,1]$. \footnote{Suppose $\Gamma_1$ contains a limit point $a$ of the sequence $\{\gamma_n(t_0)\}_{n=1}^{\infty}$, but not a limit point $b$ of the sequence $\{\gamma_n(t_1)\}_{n=1}^{\infty}$ for some $t_0, t_1 \in [0,1]$. Let $\mathcal{O}_1$ be an open neighborhood of $\Gamma_1$ that separates it from $\Gamma \setminus \Gamma_1$. Then one could find a sequence of paths $\{\gamma_k\}_{k=1}^{\infty}$, with endpoints converging to $a$ and $b$ respectively, where $b$ lies outside of $\mathcal{O}_1$. This would yield a limit point on the boundary of $\mathcal{O}_1$, which is a contradiction.} This completes the proof. 
    
\end{proof}

Theorem 2 gives a new proof (and generalization) of the mountain-climbing lemma, which originally states as follows: \footnote{There is a more general version of this lemma for continuous functions $f_1$ and $f_2$ without 'valleys'. However this formulation doesn't immediately follows from our results, since we can not guarantee the existence of a path.}:

\begin{lemma}[Mountain-climbing lemma]
    Let $f_1, f_2 \colon [0,1] \to [0,1]$ be two continuous piecewise linear functions with $f_1(0) = f_2(0) = 0$ and $f_1(1) = f_2(1) = 1$. Then there exist two continuous piecewise linear functions $g_1, g_2\colon [0,1] \to [0,1]$ such that $g_1(0) = g_2(0) = 0$, $g_1(1) = g_2(1) = 1$, and $f_1(g_1(t)) = f_2(g_2(t))$ for each $t \in [0,1]$.
\end{lemma}

\begin{proof}
    Define a 'mountain' as the graph of a function $F\colon[0,2] \to \mathbb{R}$, formed by $f_1$ followed by the reflected copy of $f_2$ (with respect to the $x$-axis). Taking the reflection of $F$ (with respect to the $y$-axis) we obtain a circle. Study the projection onto the $y$-axis. In case of a 'mountain' in general position (when there are no valleys or highest peaks at the same level), Lemma 2 immediately follows from Theorem 1, otherwise from Theorem 2. 
\end{proof}

From Theorem 3 we obtain the following $n$-dimensional extension of mountain-climbing lemma. 

\begin{theorem}[Mountain-climbing on a sphere]
    Let $f_1, ..., f_n$ be continuous functions ('climatic time-independent conditions') on the standard Euclidean sphere $\mathbb{S}^n$. Then there exists a pair of antipodal points with identical values such that 'travelers', starting from those points, can move and meet
    while, at each moment of their journey, experiencing matching 'climatic conditions' up to an arbitrarily small constant.
\end{theorem}

\section{Applications.}

\begin{theorem}[Generalization of Lusternik--Schnirelmann theorem]
    Let $M$ be a closed triangulable manifold of dimension $n \geq 1$. Suppose $M$ is covered by $n + 1$ closed sets $A_1, \dots, A_{n+1}$. Then for arbitrarily small open neighborhoods of $A_k$, which we denote by $\mathcal{E}_k$, and for any distant relation on $M$, there exists a pair of distant points (called 'travelers') such that travelers, starting from those points, can meet, along the way remaining in the same neighborhoods at each moment of time. 
    That is, there exists a path $\gamma \colon [0,1] \to M \times M$ connecting a pair of distant points to a pair of identical points such that for each $t \in [0,1]$, the points in $\gamma(t)$ belong to the same open sets $\mathcal{E}_k$ (where $k$ may change along the path).
\end{theorem}

\begin{proof}
    Let $d$ be any metric on $M$, which is compatible with the topology of $M$. Denote by $d_k$ the distance function to the closure $A_k$ for each element of the cover. Setting $\psi(x) = (d_1(x), \dots, d_{n+1}(x))$ for $x \in M$, we obtain a continuous map $\psi\colon M \to \mathbb{R}^{n+1}$. Observe that for each $x$ at least one distant function is equal to zero, hence taking the composition of $\psi$ with projection to the $n$-dimensional subspace orthogonal to $(1, \dots, 1)$, we obtain a continuous map $f \colon M \to \mathbb{R}^n$ with the same set of $f$-neighbors as $\psi$. Applying Theorem 3 we obtain a connected component $\mathcal{S}$ which is covered by sets $A_k \times A_k$ in $M \times M$. Deform this component to a path, which is covered by open neighborhoods $\mathcal{E}_k \times \mathcal{E}_k$, $k = 1, \dots, n+1$. 
\end{proof}

Next we give a qualitative generalization of Tucker's lemma stated for $n$-dimensional spheres with $n \geq 1$. 

\begin{definition}
Let $M$ be a closed triangulable manifold of dimension $n \geq 1$ with a free involution $T$. 
A triangulation of $M$ is called a \emph{Tucker triangulation} if its vertices are labeled by elements of the set $\{-1, +1, \dots, -n, +n\}$ and the triangulation is antisymmetric with respect to $T$. \footnote{A vertex labeled $k$ is mapped to the vertex labeled $-k$ under the $T$-action. We call such vertices antipodal.}
An edge of a Tucker triangulation is called a \emph{complementary} if the sum of the labels of its endpoints is zero. 
A pair of points in $M$ is called a \emph{Tucker pair} if either the sum of the labels of the corresponding vertices is zero, or both points are midpoints complementary edges. Finally, by a \emph{path} in a triangulation we mean sequence of vertices $(v_0, v_1, \dots, v_T)$ such that for each 'moment' $t$, either $v_{t+1}$ and $v_t$ are adjacent (i.e., connected by an edge), or $v_{t+1} = v_t$.
\end{definition}

\begin{theorem}[Generalization of Tucker's lemma.]
    Let $n$ be a positive integer, let $\mathbb{S}^n$ be the standard $n$-dimensional Euclidean sphere, equipped with a Tucker triangulation, in Euclidean $(n+1)$-space. Then, there exists a pair of paths $(v_0, v_1, \dots, v_T)$ and $(w_0, w_1, \dots, w_T)$ in the triangulation such that:

    \begin{enumerate}[label=\arabic*.]
        \item The paths start at antipodal vertices and end at vertices connected by a complementary edge;
        \item For each $t$, the vertices $v_t$ and $w_t$ form a Tucker pair.
    \end{enumerate}

\end{theorem}

\begin{proof}
    
    Let $L(v)$ denote a label of a vertex $v$ in the triangulation. Define a piecewise linear map $f$ of $M$
    to the standard $n$-dimensional cross-polytope $C_n$ (with vertices of type $[0, \dots, \pm 1, \dots, 0]$) as an affine extension of a vertex map defined by $f(v) = [0, \dots, \operatorname{sgn}(L(v)), \dots, 0]$, where the nonzero entry $\operatorname{sgn}(L(v))$ is placed at position $|L(v)|$.  
    
    
    Consider a sequence of piecewise linear maps in general position $\{f_k\colon M \to \mathbb{R}^n\}_{k=1}^{\infty}$, which converges to $f$. Since $T$-action gives a well-defined distant relation on $\mathbb{S}^{n}$, by Theorem 1 there exists a path $\gamma_k$ of $f_k$-neighbors in $M \times M$ which starts at a pair of antipodal points, and ends at a pair of identical points. It follows that $\gamma_k$ can be chosen to be piecewise linear, with the number of vertices uniformly bounded. Choosing a convergent subsequence \footnote{The paths $\gamma_k$ should also be equipped with an appropriate parametrization, for example, by arc length with respect to the metric induced from $\mathbb{R}^{n+1}\times\mathbb{R}^{n+1}$.} from $\{\gamma_k\}_{k=1}^{\infty}$ we obtain a peace-wise linear path $\gamma$ of $f$-neighbors $(a(t),b(t)) \in M \times M$ satisfying the same conditions, where $t \in [0,1]$.
    
    It's clear that trajectories $a(t)$ and $b(t)$ lie in some chains of incident (or adjacent) simplices (of perhaps varying dimension) of the induced by $f$ triangulation of $M$ such that images of corresponding simplices coincide ($f$ is simplicial in the induced triangulation). Now it's easy to deform paths $(a(t), b(t))$ into paths of $f$-neighbors $(\tilde{a}(t), \tilde{b}(t))$, which travel along edges of the induced triangulation by $f$: the paths start at antipodal midpoints of complementary edges and end at the same vertex. The statement follows if we take the path $(\tilde{a}(1-t), T(\tilde{b}(1-t)))$.
    
    
    

\end{proof}

\bmhead{Acknowledgements}

This research project was started during the Summer Research Program 2025 organized by the Laboratory of Combinatorial and Geometric Structures at MIPT. This research was supported by the Ministry of Science and Higher Education of the Russian Federation, agreement 075-15-2025-344 date 29/04/2025.

\end{document}